\documentclass[a4paper,10pt, draft]{article}
\usepackage[english]{babel}
\usepackage{t1enc}
\usepackage{amsfonts}
\usepackage{amsmath,amsthm,amssymb,esint}
\usepackage{color}
\usepackage{authblk}
\usepackage{graphicx}
\usepackage{tikz}
\usetikzlibrary{calc,through,backgrounds}
\usepackage{mathrsfs} 
\usepackage{multirow}
\usetikzlibrary{arrows}

\newtheorem{thm}{Theorem}[section]
\newtheorem{coro}[thm]{Corollary}

\newtheorem{lem}[thm]{Lemma}

\theoremstyle{remark}

\theoremstyle{definition}

\newcommand{\imply}{\ensuremath{\rightarrow}}

\newcommand{\yy}{\ensuremath{\wedge}}

\newcommand{\RR}{\ensuremath{\mathbb{R}}}

\newcommand{\menos}{\symbol{92}}

\newcommand{\embeds}{\ensuremath{\hookrightarrow}}

\allowdisplaybreaks 

\begin{document}

\title{A trace theorem for Besov functions in spaces of homogeneous type}

\vskip 0.3 truecm

\author{Miguel Andrés Marcos \thanks{The author was supported by Consejo Nacional
de Investigaciones Científicas y Técnicas, Agencia Nacional de Promoción Científica y Tecnológica and Universidad
Nacional del Litoral.\newline \indent Keywords and phrases:
Besov Spaces, Spaces of Homogeneous Type, Trace Theorem, Extension Theorem, Restriction Theorem, Interpolation
\newline \indent 2010 Mathematics Subject Classification: Primary
43A85.\newline }}
\affil{\footnotesize{Instituto de Matemática Aplicada del Litoral (CONICET-UNL)\\ Departamento de Matemática (FIQ-UNL)}}
\date{\vspace{-0.5cm}}

\maketitle

\begin{abstract}
The aim of this paper is to prove a trace theorem for Besov functions in the metric setting, generalizing a known result from A. Jonsson and H. Wallin in the Euclidean case.  We show that the trace of a Besov space defined in a `big set' $X$ is another Besov space defined in the `small set' $F\subset X$. The proof is divided in three parts. First we see that Besov functions in $F$ \textsl{are} restrictions of functions of the same type (but greater regularity) in $X$, that is we prove an \textsl{Extension theorem}. Next, as an auxiliary result that can also be interesting on its own, we show that the interpolation between certain potential spaces gives a Besov space. Finally, to obtain that Besov functions in $X$ \textsl{can} in fact be restricted to $F$, a \textsl{Restriction theorem}, we first prove that this result holds for functions in the potential space, and then by the interpolation result previously shown, it must hold in the Besov case. For the interpolation and restriction theorems, we make additional assumptions on the spaces $X$ and $F$, and on the order of regularity of the functions involved.
\end{abstract}

\setlength{\parskip}{10pt}

\section{Introduction}
%
%

Given a set $X$ and a subset $F$, it is natural to question whether functions from a certain Banach space $B_1(X)$ defined in $X$ leave a \textsl{trace} in $F$. This means to characterize the space $B_2(F)$ of functions defined in $F$ such that these consist exactly of functions in $B_1(X)$ \textsl{restricted} to $F$. Following \cite{JW}, this is denoted
\begin{align*}B_1(X)|_F=B_2(F),\end{align*}
and means there exist two bounded linear operators: a restriction operator $\mathcal{R}: B_1(X)\imply B_2(F)$ and an extension operator $\mathcal{E}:B_2(F)\imply B_1(X)$ such that $\mathcal{R}f=f|_F$ and $\mathcal{E}g|_F=g$, the precise meaning of these `pointwise restrictions' described below.

In the Euclidean context, for the case of Sobolev spaces $H^k$ in $\RR^n$, functions can be restricted to subspaces of smaller dimension, obtaining again Sobolev functions. For instance
\begin{align*}H^k(\RR^{n+1})|_{\RR^n}=H^{k-1/2}(\RR^n).\end{align*}

As shown by A. Jonsson and H. Wallin in \cite{JW}, these results can be generalised considering Besov spaces, proving Besov functions in $\RR^n$ leave a Besov trace in $d$-sets $F$:
\begin{align*}\Lambda_\alpha^{p,q}|_F=B^\beta_{p,q}(F),\end{align*}
with $\beta=\alpha-\frac{n-d}{p}$.

A $d$-set $F$ is a closed subset of $\RR^n$ such that there exists a measure $\mu$ supported in $F$ with $\mu(B(x,r))\sim r^d$ for balls $B(x,r)$. This measure can be assumed to be Hausdorff $d$-dimensional measure restricted to $F$. 

Besov spaces have different definitions, depending on the underlying space considered. In \cite{S}, Stein defines Besov spaces in $\RR^n$ $\Lambda_\alpha^{p,q}$ for $0<\alpha<1$ and $1\leq p,q\leq\infty$ as those $f\in L^p$ with
\begin{align*}\|f\|_p+\left(\int_0^\infty (t^{-\alpha}\omega_p f(t))^q\frac{dt}{t}\right)^{1/q}<\infty,\end{align*}
where $\omega_p f(t)=\sup_{|h|<t}\|\Delta_h f\|_p$ is the modulus of continuity. This is usually the definition assumed when dealing with Besov spaces in $\RR^n$.

For $d$-sets, Jonsson and Wallin define the Besov space $B^{p,q}_\alpha(F)$ (for the case $0<\alpha<1$) as those functions $f$ for which there exist a sequence $(f_k)$ in $L^p(F)$ and another one $(a_k)$ in $l^q$ with
\begin{itemize}
	\item $\|f-f_k\|_{L^p(F)}\leq 2^{-k\alpha}a_k$;
	\item $\left(\int_F\fint_{B(s,2^{-k})}|f_k(s)-f_k(t)|^pd\mu(t)d\mu(s)\right)^{1/p}\leq 2^{-k\alpha}a_k$
\end{itemize}
where $\mu$ denotes Hausdorff $d$-dimensional measure restricted to $F$.

In \cite{JW}, the trace theorem is divided in two parts: an extension theorem (i.e. the existence of the operator $\mathcal{E}:B^\beta_{p,q}(F)\imply \Lambda_\alpha^{p,q}$), and a restriction theorem (the existence of $\mathcal{R}: \Lambda_\alpha^{p,q}\imply B^\beta_{p,q}(F)$). 

To define the extension operator, the main tool consists in partitioning the complement of $F$ in \textsl{Whitney cubes}, with diamenter comparable to its distance to $F$, and building with those cubes a smooth partition of unity. To prove that $\mathcal{E}f:X\imply\RR$ is actually an extension of $f:F\imply\RR$, they show that the restriction
\begin{align*}\mathcal{E}f|_{F}(t)=\lim\limits_{r\imply 0}\fint_{B(t,r)}\mathcal{E}f(x)dx\end{align*}
exists and equals $f(t)$ for $\mu$-almost every $t\in F$ (this clearly coincides with the pointwise restriction of $\mathcal{E}f$ when it is continuous).

The strategy described by Jonsson and Wallin to prove the restriction theorem is to use the boundedness properties of the kernel $G_\alpha$ of the Bessel potential $\mathcal{J}_\alpha$ to guarantee the existence of a bounded operator $\mathcal{R}$ from the Sobolev potential space $\mathcal{L}^{\alpha,p}$ and a Besov-like space $l^\beta_\infty(A)$, with $A=L^p(|s-t|^{-d},d\mu(s)d\mu\emph{}(t))$, such that $\mathcal{R}f=f|_F$ for continuous functions (which are dense in $\mathcal{L}^{\alpha,p}$), and then via an interpolation argument conclude the result
\begin{align*}\begin{array}{ccc}
	\left(\mathcal{L}^{\alpha_1,p},\mathcal{L}^{\alpha_2,p}\right)_{\theta,q} & \stackrel{\mathcal{R}}{\longrightarrow} & \left(l^{\beta_1}_\infty(A),l^{\beta_2}_\infty(A)\right)_{\theta,q}\\
	\parallel & & \parallel\\
	 \Lambda_\alpha^{p,q}(\RR^n) & \stackrel{\mathcal{R}}{\longrightarrow} & l^\beta_q(A)
\end{array}
\end{align*}
where $l^\beta_q(A)\cong B_\beta^{p,q}(F)$.


In this paper we will show the result
\begin{align*}B^\alpha_{p,q}(X)|_F=B^\beta_{p,q}(F)\end{align*}
in spaces of homogeneous type, proving the existence of the two operators
\begin{align*}\mathcal{E}: B^\beta_{p,q}(F)\imply B^\alpha_{p,q}(X), \quad \mathcal{R}:B^\alpha_{p,q}(X)\imply B^\beta_{p,q}(F)\end{align*}
for $1\leq p,q\leq\infty$ and $0<\beta=\alpha-\frac{\gamma}{p}$ for a certain $\gamma$, for Besov functions defined through a modulus of continuity (see Section \ref{espbes} for the precise definition). 

In Section 2, we mention all the preliminar results we need. Using the ideas found in \cite{JW}, in Section 3 we prove an extension theorem in a general setting, as Whitney's lemma holds in this context (see \cite{A}). 

For the restriction part, we consider Ahlfors regular spaces. Using the potential-Sobolev spaces $L^{\alpha,p}$ constructed in \cite{M}, in Section 4 we show that for small orders of regularity, the interpolation between two potential spaces gives a Besov space, generalizing the result from $\RR^n$ (see \cite{P}). Finally, in Section 5 we prove first a restriction theorem for potential functions, and then conclude (by the interpolation result) that for small orders of regularity, the result holds for Besov functions.

\section{Preliminaries}

In this section we describe the geometric setting and basic properties from harmonic analysis on spaces of homogeneous type needed to prove our results.

\subsection{The geometric setting}\label{geom}

We say $(X,\rho,m)$ is a \textbf{space of homogeneous type} if $\rho$ is a quasi-metric on $X$ and $m$ a measure such that balls and open sets are measurable and that satisfies the \textbf{doubling property}: there exists a constant $C>0$ such that
\begin{align*}m(B_\rho(x,2r))\leq Cm(B_\rho(x,r))\end{align*}
for each $x\in X$ and $r>0$.

If $m(\{x\})=0$ for each $x\in X$, by \cite{MS} there exists a metric $d$ giving the same topology as $\rho$ and a number $N>0$ such that $(X,d,m)$ satisfies
\begin{align}\label{ahlf}m(B_d(x,2r))\sim r^N\end{align}
for each $x\in X$ and $0<r<\text{diam}(X)$.

Spaces that satisfy condition \ref{ahlf} are called \textbf{Ahlfors $N$-regular}. Besides $\RR^n$ (with $N=n$), examples include self-similar fractals such as the Cantor ternary set or the Sierpi\'nski gasket.

In Section 3 we will consider $(X,d,m)$ a space of homogeneous type such that $d$ is a metric, and $F\subset X$ a closed set with $m(F)=0$. We will also consider a Borel measure $\mu$ defined in $F$ such that $\mu$ is doubling for balls centered in $F$, and such that the following holds
\begin{align*}\frac{m(B)}{\mu(B)}\sim r^{\gamma}.\end{align*}

Observe that the cases in which $X$ and $F$ are Ahlfors $N$-regular and $d$-regular, respectively, satisfy this quotient relation with $\gamma=N-d$. As another example satisfying the above, if $F$ is a doubling measure space and $Y$ is Ahlfors $\gamma$-regular, the spaces $X=F\times Y$ and $F$ also satisfy this condition, if we take the product metric and the product measure for $X$ (and if we identify $F$ with the subset of $X$ $F\times \{y_0\}$ for some $y_0\in Y$).

In Sections 4 and 5, we will assume $(X,d,m)$ to be Ahlfors $N$-regular with $m(X)=\infty$ and $(F,d,\mu)$ to be Ahlfors $d$-regular (where the metric $d$ in the second case means $d|_{F\times F}$). Using the same notation $d$ for the distance function and for the regularity of the space will not cause any confusion, as the meaning will be clear from context.

To mention an interesting example of an Ahlfors space satisfying $m(X)=\infty$, we can modify the Sierpi\'nski gasket $T$ by taking dilations (powers of 2): $\tilde{T}=\cup_{k\geq 1}2^kT$. This $\tilde{T}$ preserves some properties of the original triangle, including the Ahlfors character.

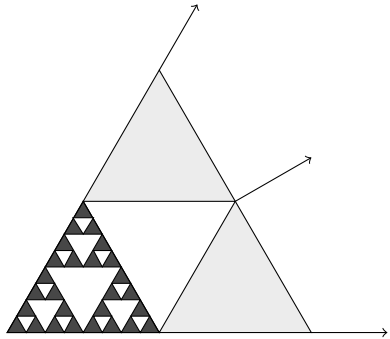
\begin{figure}[h!]\begin{center}\begin{tikzpicture}[scale=2]
\draw[fill=gray, fill opacity=0.15] (-0.5,0)--(1.5,0)--(0.5,1.74)--cycle;
\draw[fill=white] (0.5,0)--(0,0.87)--(1,0.87)--cycle;
\draw[<-] (2,0)--(1.5,0);
\draw[<-] (0.75,2.175)--(0.5,1.74);
\draw[<-] (1.5,1.16)--(1,0.87);
\draw[fill=black, fill opacity=0.7] (-0.5,0)--(0.5,0)--(0,0.87)--cycle;
\draw[fill=white] (0,0)--(0.25,0.435)--(-0.25,0.435)--cycle;
\draw[fill=white] (-0.25,0)--(-0.125,0.2175)--(-0.375,0.2175)--cycle;
\draw[fill=white] (0.25,0)--(0.375,0.2175)--(0.125,0.2175)--cycle;
\draw[fill=white] (0,0.435)--(0.125,0.6525)--(-0.125,0.6525)--cycle;
\draw[fill=white] (-0.375,0)--(-0.4375,0.10875)--(-0.3125,0.10875)--cycle;
\draw[fill=white] (0.375,0)--(0.4375,0.10875)--(0.3125,0.10875)--cycle;
\draw[fill=white] (-0.125,0)--(-0.1875,0.10875)--(-0.0675,0.10875)--cycle;
\draw[fill=white] (0.125,0)--(0.1875,0.10875)--(0.0675,0.10875)--cycle;
\draw[fill=white] (-0.25,0.2175)--(-0.3125,0.32625)--(-0.1875,0.32625)--cycle;
\draw[fill=white] (0.25,0.2175)--(0.3125,0.32625)--(0.1875,0.32625)--cycle;
\draw[fill=white] (-0.125,0.4350)--(-0.1875,0.54375)--(-0.0675,0.54375)--cycle;
\draw[fill=white] (0.125,0.4350)--(0.1875,0.54375)--(0.0675,0.54375)--cycle;
\draw[fill=white] (0,0.6525)--(-0.0675,0.76125)--(0.0675,0.76125)--cycle;
\end{tikzpicture}\end{center}\end{figure}

\subsection{Whitney's lemma}

An essential tool in proving the extension theorem is a Whitney-type covering lemma that remains valid in the metric setting. The one we will use is an adaptation of a theorem found in \cite{A}. As this result is purely geometrical, it remains valid in a much general scenario as follows. 

We say a (quasi)metric space $(X,d)$ has the \textbf{weak homogeneity property} if there exists $n$ such that for each $r>0$ and each $E\subset X$ satisfying $d(x,y)\geq r/2$ for $x,y\in E$, then we have that $\#(E\cap B(x_0,r))\leq n$ for any $x_0\in X$. Clearly spaces of homogeneous type satisfy the weak homogeneity property.

\begin{lem}[\textbf{Whitney type covering + partition of unity}]\label{whit} Let $F$ be a closed subset of a metric space $(X,d)$ that satisfies the weak homogeneity property. Then there exists a (countable) collection $\{B_i=B(x_i,r_i)\}_i$ of balls satisfying
\begin{enumerate}
	\item $\{B_i\}$ are pairwise disjoint;
	\item $0<d(x_i,F)<\frac{1}{2}$ for each $i$ and $\{x:0<d(x,F)<1\}\subset\cup_i 3B_i$;
	\item $6 r_i\leq d(x,F)\leq 18r_i$ for each $x\in 6B_i$, for each $i$;
	\item for each $i$ there exists $t_i\in F$ satisfying $d(x_i,t_i)<18 r_i$;
	\item if $\Omega=\cup_i 6B_i$, then $F\cap\Omega=\emptyset$ and $\partial F\subset \overline{\Omega}$;
	\item there exists $M>0$ such that for each $i$, $\#\{j: 6B_i\cup 6B_j\neq\emptyset\}\leq M$.
\end{enumerate}
Furthermore, there exists a collection $\{\varphi_i\}_i$ of real functions satisfying
\begin{enumerate}
	\item $3B_i\subset supp\varphi_i\subset 6B_i$;
	\item $0\leq\varphi_i\leq1$ and $0\leq\sum_i\varphi_i\leq 1$;
	\item $\sum_i\varphi_i(x)=1$ if $0<d(x,F)<1$;
	\item $\sum_i\varphi_i\equiv 0$ outside of $\Omega$;
	\item $\varphi_i\equiv 1$ in $B_i$;
	\item for each $i$, $|\varphi_i(x)-\varphi_i(y)|\leq \frac{C}{r_i}d(x,y)$ with $C$ independent of $i$. 
\end{enumerate}
\end{lem}

We will also need the following result.

\begin{lem}\textbf{Bounded overlap.} Let $(X,d)$ be a metric space with the weak homogeneity property and let $1\leq a<b$, $\kappa>1$. There exists a constant $C$ such that, if $\{B_i=B(x_i,r_i)\}_i$ is a family of disjoint balls, and $r>0$, 
$$\sum_{i: ar\leq r_i\leq br}\chi_{\kappa B_i}\leq C.$$\end{lem}

\subsection{Besov spaces}\label{espbes}

Throughout this work we use a definition of Besov spaces via modulus of continuity, as can be found in \cite{GKS}. For a definition based in aproximations of the identity see \cite{HS}, and for an equivalence between both see \cite{MY}.

For $1\leq p<\infty$ and $f\in L^1_{loc}$, its $p$-modulus of continuity of a function $f$ is defined as
\begin{align*}E_pf(t)=\left(\int_X\fint_{B(x,t)}|f(y)-f(x)|^pdm(y)dm(x)\right)^{1/p}\end{align*}
for $t>0$. With this, its Besov norm
\begin{align*}\|f\|_{B^\alpha_{p,q}}=\|f\|_p+\left(\int_0^\infty \left(t^{-\alpha}E_pf(t)\right)^q\frac{dt}{t}\right)^{1/q}\end{align*}
(for $1\leq q<\infty$, and the usual modification for $q=\infty$) and then the Besov space $B^\alpha_{p,q}$ consists of those functions $f$ with finite norm.

Observe that, as $E_pf(t)\leq C\|f\|_p$, the integral from $0$ to $\infty$ can be restricted to an integral from $0$ to $1$, giving equivalent norms. If $m$ is doubling, it can also be discretized as follows:
\begin{align*}\|f\|_{B^\alpha_{p,q}}\sim \|f\|_p+\left(\sum_{k\geq 1} \left(2^{k\alpha}E_pf(2^{-k})\right)^q\right)^{1/q}.\end{align*}

Additionally, this definition gives equivalent norms to the ones described for $\RR^n$ and $d$-sets in the introduction.

\subsection{Singular Integrals}

In Ahlfors $N$-regular spaces, the following version of the $T1$ theorem hold (see for instance \cite{Ga}). We require $m(X)=\infty$.

A continuous function $K:X\times X\menos\Delta\imply\RR$ (where $\Delta=\{(x,x):x\in X\}$) is a standard kernel if there exist constants $0<\eta\leq 1$, $C>0$ such that
\begin{itemize}
	\item $|K(x,y)|\leq Cd(x,y)^{-N}$;
	\item for $x\neq y$, $d(x,x')\leq cd(x,y)$ (with $c<1$) we have
	\begin{align*}|K(x,y)-K(x',y)|\leq Cd(x,x')^\eta d(x,y)^{-(N+\eta)};\end{align*}
	\item for $x\neq y$, $d(y,y')\leq cd(x,y)$ (with $c<1$) we have
	\begin{align*}|K(x,y)-K(x,y')|\leq Cd(y,y')^\eta d(x,y)^{-(N+\eta)}.\end{align*}
\end{itemize}

Let $C_c^\gamma$ denote the space of Lipschitz-$\gamma$ functions with compact support. A linear continuous operator $T:C_c^\gamma\imply (C_c^\gamma)'$ for $0<\gamma\leq 1$ is a singular integral operator with associated standard kernel $K$ if it satisfies
\begin{align*}\langle Tf,g\rangle = \iint K(x,y)f(y)g(x)dm(y)dm(x),\end{align*}
for $f,g\in C_c^\gamma$ with disjoint supports. If a singular integral operator can be extended to a bounded operator on $L^2$ it is called a Calderón-Zygmund operator or CZO.

Every CZO is bounded in $L^p$ for $1<p<\infty$, of weak type $(1,1)$, and bounded from $L^\infty$ to $BMO$.

The $T1$ theorem characterizes CZO's. We say that an operator is weakly bounded if
\begin{align*}|\langle Tf,g\rangle|\leq Cm(B)^{1+2\gamma/N}[f]_\gamma[g]_\gamma,\end{align*}
for $f,g\in C_c^\gamma(B)$, for each ball $B$.

\begin{thm}[\textbf{$\mathbf{T1}$ theorem}]\label{tedeuno} Let $T$ be a singular integral operator. Then $T$ is a CZO if and only if $T1,T^*1\in BMO$ and $T$ is weakly bounded.\end{thm}

\subsection{Approximations of the identity}\label{aprox}

In Ahlfors spaces with $m(X)=\infty$, Coifman-type aproximations of the identity $(S_t)_{t>0}$ can be constructed (see for instance \cite{GSV} or \cite{HS}). Let $s$ be the kernel associated to $(S_t)$, meaning
\begin{align*}S_tf(x)=\int_X f(y)s(x,y,t)dm(y)\end{align*}
and define
\begin{align*}q(x,y,t)=-t\frac{d}{dt}s(x,y,t), \hspace{1cm} Q_tf(x)=\int_X f(y)q(x,y,t)dm(y).\end{align*}

This satisfies the following:
\begin{enumerate}
	\item $Q_t1\equiv0$ for all $t>0$;
	\item $q(x,y,t)=q(y,x,t)$ for $x,y\in X$, $t>0$;
	\item $|q(x,y,t)|\leq C/t^N$ for $x,y\in X$, $t>0$;
	\item $q(x,y,t)=0$ if $d(x,y)>4t$;
	\item $|q(x,y,t)-q(x',y,t)|\leq C'\frac{1}{t^{N+1}}d(x,x')$;
	\item $Q_t: L^p\imply L^p$;
\end{enumerate}

In this case, the following version of Calderon's reproducing formula holds:
\begin{align*}f=\int_0^\infty \tilde{Q}_sQ_sf\frac{ds}{s}\end{align*}
for $\tilde{Q}_sf=\int_X \tilde{q}(\cdot,y,t)f(y)dm(y)$, where $\tilde{q}$ satisfies, for any $0<\xi<1$,
\begin{itemize}
	\item $|\tilde{q}(x,y,t)|\leq C\frac{t^\xi}{(t+d(x,y))^{N+\xi}}$;
	\item $\int_X \tilde{q}(x,y,t)dm(y)=\int_X\tilde{q}(z,x,t)dm(z)=0$ for each $x$ and $t>0$;
	\item for $d(x,x')<C(t+d(x,y))$,
		\begin{align*}|\tilde{q}(x,y,t)-\tilde{q}(x',y,t)|\leq C\frac{t^\xi d(x,x')^\xi}{(t+d(x,y))^{N+2\xi}},\end{align*}
\end{itemize}
see for instance \cite{HS} for a discrete version.

\subsection{Bessel Potentials in Ahlfors spaces}\label{bess}

In $\RR^n$, Bessel potentials $\mathcal{J}_\alpha=(I-\Delta)^{-\alpha/2}$ (where $\Delta$ is the Laplacian) and its associated Sobolev potential spaces $\mathcal{L}^{\alpha,p}=\mathcal{J}_\alpha(L^p)$ play a crucial role in the proof of the restriction theorem in \cite{JW}. The two key aspects are that Sobolev functions can be restricted to the subset $F$ and that interpolation between two Sobolev spaces gives a Besov space. 

The restriction of functions in $\mathcal{L}^{\alpha,p}$ to $F$ requires appropiate boundedness properties of the kernel of $\mathcal{J}_\alpha$, while the interpolation result uses the fact that $\mathcal{J}_\alpha^{-1}\approx I+\mathcal{D}_\alpha$, where $\mathcal{D}_\alpha=(-\Delta)^{\alpha/2}$ (see \cite{P}).

Potential-type Sobolev spaces are defined and studied in \cite{M} via the kernel
\begin{align*}k_\alpha(x,y)=\int_0^\infty \frac{\alpha t^\alpha}{(1+t^\alpha)^2}s(x,y,t)\frac{dt}{t},\end{align*}
where $s$ is the approximation of the identity described in Section \ref{aprox}. This kernel satisfies 
\begin{itemize}
	\item $k_\alpha\geq 0$;
	\item $k_\alpha(x,y)=k_\alpha(y,x)$
	\item $k_\alpha(x,y)\leq C d(x,y)^{-(N-\alpha)}$;
	\item $k_\alpha(x,y)\leq C d(x,y)^{-(N+\alpha)}$ if $d(x,y)\geq4$;
	\item $|k_\alpha(x,z)-k_\alpha(y,z)|\leq C d(x,y)(d(x,z)\yy d(y,z))^{-(N+1-\alpha)}$;
	\item $|k_\alpha(x,z)-k_\alpha(y,z)|\leq C d(x,y)(d(x,z)\yy d(y,z))^{-(N+1+\alpha)}$ if $d(x,z)\geq 4$ and $d(y,z)\geq 4$;
\end{itemize}
and also the following integral properties
\begin{itemize}
	\item $\int_X k_\alpha(x,z)dm(z)=\int_X k_\alpha(z,y)dm(z)=1$ $\forall x,y$;
	\item for $q(N-\alpha)<N<q(N+\alpha)$, there exists $C>0$ such that for all $x$,
		\begin{align*}\int_X k_\alpha(x,y)^qdm(y)\leq C;\end{align*}
	\item for $q(N-\alpha)<N<q(N-\alpha+1)$, there exists $C>0$ such that for all $x,y$,
		\begin{align*}\int_X |k_\alpha(x,z)-k_\alpha(y,z)|^qdm(z)\leq Cd(x,y)^{N-q(N-\alpha)}.\end{align*}
\end{itemize}

With this kernel, the Bessel-type potential is defined as
\begin{align*}J_\alpha g(x)=\int_X g(y)k_\alpha(x,y)dm(y),\end{align*}
or equivalently as
\begin{align*}J_\alpha f(x)=\int_0^\infty \frac{1}{1+t^{-\alpha}}Q_tf(x)\frac{dt}{t},\end{align*}
for $Q_t$ also as in Section \ref{aprox}, and turns out to be bounded in $L^p$ for $1\leq p\leq \infty$. The potential space $L^{\alpha,p}=J_\alpha(L^p)$ is proved Banach when equiped with the norm
\begin{align*}\|f\|_{\alpha,p}=\|f\|_p+\inf_g\|g\|_p\end{align*}
where the infimum is taken over all $g$ with $f=J_\alpha g$.

These spaces satisfy, among other things, that continuous functios are dense in $L^{\alpha,p}$ for $\alpha>0$ and $1\leq p\leq q$, and also the embedding $L^{\alpha,p}\embeds B^\alpha_{p,\infty}$ for $0<\alpha<1$, $1\leq p\leq\infty$.

With the fractional derivative $D_\alpha$ defined as
\begin{align*}D_\alpha f(x)=\int_X (f(x)-f(y)) n_\alpha(x,y)dm(y)\end{align*}
with
\begin{align*}n_\alpha (x,y)=\int_0^\infty \alpha t^{-\alpha}s(x,y,t)\frac{dt}{t},\end{align*}
or equivalently as
\begin{align*}D_\alpha f(x)=\int_0^\infty t^{-\alpha}Q_tf(x)\frac{dt}{t},\end{align*}
in \cite{M} it is proved there exists $0<\alpha_0<1$ such that the following theorem holds.

\begin{thm}\label{caract}Let $1<p<\infty$ and $0<\alpha<\alpha_0$. Then
\begin{align*}f\in L^{\alpha,p}\text{ if and only if } f,D_\alpha f\in L^p,\end{align*}
Furthermore,
\begin{align*}\|f\|_{\alpha,p}\sim \|(I+D_\alpha)f\|_p.\end{align*}
\end{thm}

For this range of $\alpha$, in \cite{M} it is also shown that in $\RR^n$ this potential spaces $L^{\alpha,p}$ coincide with the classical ones $\mathcal{L}^{\alpha,p}$.

\subsection{Interpolation}\label{jotaka}

The following results and definitions, concerning interpolation between Banach spaces, is based in \cite{BS}.

Given $X, Y$ Banach, both embedded in the same topological vector space, we define in $X+Y$ and $X\cap Y$ the norms
\begin{align*}\|f\|_{X+Y}=\inf \limits_{f=g+h} \|g\|_X+\|h\|_Y,\end{align*}
\begin{align*}\|f\|_{X\cap Y}=\max(\|f\|_X,\|f\|_Y)\end{align*}
respectively. Peetre's $K$ and $J$-functionals for these spaces $X,Y$ are defined as
\begin{align*}Kf(t)=\inf\limits_{f=g+h}\|g\|_X+t\|h\|_Y, \quad f\in X+Y;\end{align*}
\begin{align*}Jf(t)=\max(\|f\|_X,t\|f\|_Y), \quad f\in X\cap Y.\end{align*}

Now, for $0<\theta<1$, $1\leq q\leq\infty$, the interpolated space $(X,Y)_{\theta, q}$ is defined as the set of those $f$ with the following finite norm
\begin{align*}\|f\|_{\theta,q,K}=\|t^{-\theta}Kf(t)\|_{L^q((0,\infty),\frac{dt}{t})}.\end{align*}

This space satisfies
\begin{align*}X\cap Y\embeds (X,Y)_{\theta,q}\embeds X+Y,\end{align*}
and if we have pairs $X,Y$ y $X',Y'$ as before, and $T$ is bounded an linear from $X$ to $X'$ and from $Y$ to $Y'$, then it also satisfies
\begin{align*}T:(X,Y)_{\theta,q}\imply (X',Y')_{\theta,q}\end{align*}
with
\begin{align*}\|Tf\|_{\theta,q,K}\leq \|T\|_{X\imply X'}^{1-\theta}\|T\|_{Y\imply Y'}^\theta \|f\|_{\theta,q,K}.\end{align*}

One can also define the same interpolated space with the aid of the $J$-functional, as the set of those $f\in X+Y$ that have a Bochner integral decomposition
\begin{align*}f=\int_0^\infty A_sf \frac{ds}{s},\end{align*}
with $A_sf\in X\cap Y$ for each $s$, such that
\begin{align*}\|f\|_{\theta,q,J}=\inf_{A_s}\|s^{-\theta}J(A_sf)(s)\|_{L^q((0,\infty),\frac{ds}{s})},\end{align*}
where the infimum is taken over all possible decompositions of $f$.

\section{The extension theorem}

First we enunciate the theorem we want to prove in this section. 

\begin{thm}[\textbf{Extension theorem for Besov functions}]\label{exten}Let $(X,d,m)$ be a space of homogeneous type with $d$ a metric, and let $F\subset X$ be closed with $m(F)=0$. Assume $\mu$ is a nontrivial Borel measure with support $F$ which is doubling for balls centered in $F$, and that there exists $\gamma>0$ such that 
\begin{align}\label{cociente}\frac{m(B)}{\mu(B)}\sim r^{\gamma}\end{align}
for balls $B$ centered in $F$ with radius $r<2diam(F)$. 

Then there is an extension operator $\mathcal{E}$ for functions $f\in L^1_{loc}(F,\mu)$ that satisfies, for $1\leq p<\infty$ and $1\leq q\leq\infty$ and $0<\beta<1-\gamma/p$,
\begin{align*}\mathcal{E}:B_{p,q}^\beta(F,\mu)\imply B_{p,q}^\alpha(X,m),\end{align*}
where $\alpha=\beta+\gamma/p$.\end{thm}

In Section \ref{geom} some examples satisfying \ref{cociente} are included. Certainly, Theorem \ref{exten} contains the classical result in Theorem 1, Chapter VI from \cite{JW} for the case $0<\alpha<1$.

We will need the following discrete version of Hardy's Inequality, see for instance \cite{L}.

\begin{lem}\textbf{Hardy's Inequality.} Let $(b_n)$ be a sequence of nonnegative terms, $\gamma>0$ and $a>0$, then there exists $C>0$ such that
$$\sum_{n=0}^\infty 2^{-na}\left(\sum_{k=0}^n b_k\right)^{\gamma}\leq C\sum_{n=0}^\infty 2^{-na}b_n^{\gamma}.$$\end{lem}

\begin{proof}[Proof of Theorem \ref{exten}]

Let us first define the extension operator: let us take $\{B_i,\varphi_i\}_i$ as in Whitney's Lemma (Lemma \ref{whit}). If $f\in L^1_{loc}(F,\mu)$ and $x\in X$, we define
\begin{align*}\mathcal{E}f(x)=\sum_i \varphi_i(x)\fint_{19B_i}fd\mu.\end{align*}

Observe that by the properties possessed by $B_i$ and $\varphi_i$, $\mu(19B_i)>0$ for each $i$ so the constants multiplying each $\varphi_i$ make sense. Also, $\mathcal{E}f$ is continuous in $\Omega=\cup_i 6B_i$ and $\mathcal{E}f\equiv 0$ in the exterior $\Omega$, and it will not necessarily be continuous in $F\subset \partial\Omega$.

We will prove the theorem in three parts: we first check that $\mathcal{E}$ is in fact an extension operator for each $f\in L^1_{loc}(F,\mu)$. Next we prove $\mathcal{E}$ is bounded from $L^p(F,\mu)$ to $L^p(X,m)$, and we finish the theorem proving $\mathcal{E}:B_{p,q}^\beta(F,\mu)\imply B_{p,q}^\alpha(X,m)$ for $\alpha,\beta$ as in the statement of the theorem.

\textsl{\textbf{$\mathcal{E}$ is an extension operator.}}

We need to show $\mathcal{E}f|_F=f$ (remember $m(F)=0$ and that $F\subset\partial\Omega$). As mentioned in the introduction, this means that we have to prove that the function
\begin{align*}\mathcal{E}f|_{F}(t_0)=\lim\limits_{r\imply 0}\fint_{B(t_0,r)}\mathcal{E}f(x)dx\end{align*}
is well defined and equals $f(t_0)$ for $\mu$-almost every $t_0\in F$.


Let $t_0\in F$ be a $\mu$-Lebesgue point of $f$ and $0<r<1$. For $x$ with $0<d(x,F)<1$ we have $\sum_i\varphi_i(x)=1$, so
\begin{align*}\mathcal{E}f(x)-f(t_0)=\sum_i\varphi_i(x)\fint_{19B_i}(f(t)-f(t_0))d\mu(t),\end{align*}
and then
\begin{align*}|\mathcal{E}f(x)-f(t_0)|\leq \sum_i\chi_{6B_i}(x)\fint_{19B_i}|f(t)-f(t_0)|d\mu(t).\end{align*}
Let now $x\in B(t_0,r)\menos F$, then there exists $k$ such that $2^{-k}<cr$ and $d(x,F)\sim 2^{-k}$, as $t_0\in F$, and in that case
	\begin{align*}|\mathcal{E}f(x)-f(t_0)| &\leq \int_F |f(t)-f(t_0)|\left(\sum_{r_i\sim 2^{-k}}\chi_{19B_i}(t)\frac{\chi_{6B_i}(x)}{\mu(19B_i)}\right)d\mu(t)\\
	&\leq \int_{B(t_0,r+c2^{-k})} |f(t)-f(t_0)|\left(\sum_{r_i\sim 2^{-k}}\chi_{19B_i}(t)\frac{\chi_{19B_i}(x)}{\mu(19B_i)}\right)d\mu(t)\\
	&\leq \int_{B(t_0,cr)} |f(t)-f(t_0)|\left(\sum_{r_i\sim 2^{-k}}\chi_{19B_i}(t)\frac{\chi_{19B_i}(x)}{\mu(19B_i)}\right)d\mu(t);\end{align*}
now integrating over $B(t_0,r)$ (using $m(B(t_0,r)\cap F)=0$), by the bounded overlap property and the relation between the measures,
	\begin{align*}\int_{B(t_0,r)}|\mathcal{E}f(x)-f(t_0)|dm(x) &\leq\sum_{2^{-k}\leq cr}\int_{x\in B(t_0,r), d(x,F)\sim 2^{-k}}|\mathcal{E}f(x)-f(t_0)|dm(x)\\
	&\leq \sum_{2^{-k}\leq cr} \int_{B(t_0,cr)} |f(t)-f(t_0)|\left(\sum_{r_i\sim 2^{-k}}\chi_{19B_i}(t)\frac{m(19B_i)}{\mu(19B_i)}\right)d\mu(t)\\
	&\leq C\left(\sum_{2^{-k}\leq cr} 2^{-k\gamma}\right)\int_{B(t_0,cr)} |f(t)-f(t_0)|d\mu(t)\\
	&\leq Cr^\gamma\int_{B(t_0,cr)} |f(t)-f(t_0)|d\mu(t)\\
	&\leq C\frac{m(B(t_0,cr))}{\mu(B(t_0,cr))}\int_{B(t_0,cr)} |f(t)-f(t_0)|d\mu(t).\end{align*}
	
	In other words,
	\begin{align*}\fint_{B(t_0,r)}|\mathcal{E}f(x)-f(t_0)|dm(x)\leq C\fint_{B(t_0,cr)}|f(t)-f(t_0)|d\mu(t),\end{align*}
	and as the right side tends to $0$ when $r\imply 0$ for $\mu$-a.e. $t_0\in F$ (as $\mu$ is doubling and $f\in L^1_{loc}(F,\mu)$, so Lebesgue's differentiation theorem holds), the left side also tends to $0$ and $\mathcal{E}f|_F=f$ $\mu$-a.e.

\textbf{\textsl{Boundedness in $L^p$.}}

We need to check now that there exists $C>0$ such that $\|\mathcal{E}f\|_{p,m}\leq C\|f\|_{p,\mu}$ for every $f\in L^p(F,\mu)$.

We have
\begin{align*}|\mathcal{E}f(x)|=\left|\sum_i \varphi_i(x)\fint_{19B_i}fd\mu\right|\leq \sum_i \varphi_i(x)\chi_{6B_i}(x)\left(\fint_{19B_i}|f|^pd\mu\right)^{1/p},\end{align*}
and so by the discreete version of Hölder's inequality, as $\left(\sum_i \varphi_i(x)^{p'}\right)^{1/p'}\leq 1$,
\begin{align*}|\mathcal{E}f(x)|^p &\leq \sum_i \chi_{6B_i}(x)\fint_{19B_i}|f|^pd\mu \\
&\leq \int_F |f(t)|^p \left(\sum_i \chi_{19B_i}(t)\frac{\chi_{19B_i}(x)}{\mu(19B_i)}\right)d\mu(t).\end{align*}
Once again by bounded overlap,
\begin{align*}\int_X|\mathcal{E}f(x)|^pdm(x) &\leq C\int_F |f(t)|^p \left(\sum_i \chi_{19B_i}(t)\frac{m(19B_i)}{\mu(19B_i)}\right)d\mu(t)\\
&\leq C\int_F |f(t)|^p \left(\sum_k\sum_{r_i\sim 2^{-k}} \chi_{19B_i}(t)\frac{m(19B_i)}{\mu(19B_i)}\right)d\mu(t)\\
&\leq C\int_F |f(t)|^p \left(\sum_k 2^{-k\gamma}\right)d\mu(t)\leq C\|f\|^p_{p,\mu}.\end{align*}

\textbf{\textsl{Boundedness in Besov spaces.}}

To see this last part of the theorem, we need to check that there exists $C>0$ such that $[\mathcal{E}f]_{B^\alpha_{p,q}(X,m)}\leq C\|f\|_{B^\beta_{p,q}(F,\mu)}$ for every $f\in B^\beta_{p,q}(F,\mu)$. It will be enough to prove that, for $q<\infty$,
\begin{align*}\sum_{l\geq 1}2^{l\alpha q}E_p\mathcal{E}f(2^{-l})^q\leq C\|f\|_{B^\beta_{p,q}(F,\mu)}^q\end{align*}
and for the case $q=\infty$,
\begin{align*}\sup_{l\geq 1}2^{l\alpha}E_p\mathcal{E}f(2^{-l})\leq C\|f\|_{B^\beta_{p,\infty}(F,\mu)}.\end{align*}

We will first prove that, for $l\geq 1$,
\begin{align*}E_p\mathcal{E}f(2^{-l})^p &\leq\\
&\hspace{-1cm}\leq C2^{-lp}\sum_{k\leq l}2^{k(p-\gamma)}E_pf(c2^{-k})^p+C2^{-lp}\|f\|_{p,\mu}^p +C2^{-l\gamma}E_pf(c2^{-l})^p.\end{align*}

For this, we split $E_p\mathcal{E}f(2^{-l})^p$ in four parts:
\begin{align*}E_p\mathcal{E}f(2^{-l})^p 
&\leq \int_{0<d(x,F)< 2^{-l}}\fint_{B(x,2^{-l})}|\mathcal{E}f(x)-\mathcal{E}f(y)|^p dm(y)dm(x)\\
&\hspace{1cm}+\int_{2^{-l}\leq d(x,F)<\frac{3}{2}, x\in\Omega}\int_{B(x,2^{-l})\cap\Omega}\frac{|\mathcal{E}f(x)-\mathcal{E}f(y)|^p}{m(B(x,2^{-l}))}dm(y)dm(x)\\
&\hspace{1cm}+\int_{\frac{1}{2}\leq d(x,F)<\frac{3}{2}, x\in\Omega} \int_{B(x,2^{-l})\menos\Omega}\frac{|\mathcal{E}f(x)-\mathcal{E}f(y)|^p}{m(B(x,2^{-l}))}dm(y)dm(x)\\
&\hspace{1cm}+\int_{\frac{1}{2}\leq d(x,F)<\frac{3}{2}, x\not\in\Omega} \int_{B(x,2^{-l})\cap\Omega}\frac{|\mathcal{E}f(x)-\mathcal{E}f(y)|^p}{m(B(x,2^{-l}))}dm(y)dm(x)\\
&=I+II+III+IV\end{align*}
(notice that if $x\not\in\Omega$ then $\mathcal{E}f(x)=0$, and the same holds for $y$, then as $d(x,y)<2^{-l}\leq 1/2$, the terms $I$ and $II$ correspond to $x,y\in\Omega$, while in $III$ we have $y\not\in\Omega$ and in $IV$, $x\not\in\Omega$).

We start by bounding $III$, and $IV$ will be bounded in a similar way. Here we have $\varphi_i(y)=0$ for each $i$, then
\begin{align*}\mathcal{E}f(x)-\mathcal{E}f(y) &=\mathcal{E}f(x)=\sum_{i:x\in 6B_i}\varphi_i(x)\fint_{19B_i}fd\mu\\
&=\sum_{i:x\in 6B_i}(\varphi_i(x)-\varphi_i(y))\fint_{19B_i}fd\mu.\end{align*}
This way, using the Lipschitz condition of $\varphi_i$, and the fact that $d(x,y)<2^{-l}$ and $\frac{1}{2}\leq d(x,F)<\frac{3}{2}$ (so $\frac{1}{36}<r_i<\frac{1}{4}$),
\begin{align*}|\mathcal{E}f(x)-\mathcal{E}f(y)|^p &\leq C\sum_{i:x\in 6B_i}\frac{d(x,y)^p}{r_i^p}\fint_{19B_i}|f|^pd\mu\\
&\leq C2^{-lp}\int_F |f(t)|^p \sum_i \frac{\chi_{6B_i}(x)\chi_{19B_i}(t)}{\mu(19B_i)}d\mu(t),\end{align*}
and therefore
\begin{align*}III &\leq C2^{-lp}\int_F |f(t)|^p \sum_i \left(\int_{\frac{1}{2}\leq d(x,F)<\frac{3}{2}} \chi_{6B_i}(x)\frac{\int_{B(x,2^{-l})\menos \Omega}dm(y)}{m(B(x,2^{-l}))}dm(x)\right)\frac{\chi_{19B_i}(t)}{\mu(19B_i)}d\mu(t)\\
& \leq C2^{-lp}\int_F |f(t)|^p \left(\sum_{\frac{1}{36}<r_i<\frac{1}{4}}\chi_{19B_i}(t)\frac{m(6B_i)}{\mu(19B_i)}\right)d\mu(t)\\
& \leq C2^{-lp}\|f\|_{p,\mu}^p.\end{align*}

For $II$, $d(x,F)\sim 2^{-k}$ for some $k\leq l$, and we can write, using $\sum_i \varphi_i(x)-\varphi_i(y)=0$ (as $x,y\in\Omega$),
	\begin{align*}\mathcal{E}f(x)-\mathcal{E}f(y)=\sum_i (\varphi_i(x)-\varphi_i(y))\fint_{19B_i}\fint_{B(x,2^{-k})}f(s)-f(t)d\mu(t)d\mu(s),\end{align*}
	and for $x\in 3B_j$, $r_j\sim 2^{-k}$,
\begin{align*}\int_{3B_j}\fint_{B(x,2^{-l})}|\mathcal{E}f(x)-\mathcal{E}f(y)|^pdm(y)dm(x) &\leq\\
&\hspace{-5cm}\leq C\int_{3B_j}\fint_{B(x,2^{-l})} \sum_i |\varphi_i(x)-\varphi_i(y)|^p\fint_{19B_i}\fint_{B(x,2^{-k})}|f(s)-f(t)|^pd\mu(t)d\mu(s) dm(y)dm(x)\\
&\hspace{-5cm}\leq C\int_{3B_j}\fint_{B(x,2^{-l})} \sum_{i:x\vee y\in supp\varphi_i} d(x,y)^p r_i^{-p}\fint\limits_{cB_j}\fint\limits_{B(s,c2^{-k})}|f(s)-f(t)|^pd\mu(t)d\mu(s) dm(y)dm(x)\\
&\hspace{-5cm}\leq C 2^{kp}2^{-lp}\frac{m(3B_j)}{\mu(cB_j)}\int\limits_{cB_j}\fint\limits_{B(s,c2^{-k})}|f(s)-f(t)|^pd\mu(t)d\mu(s)\\
&\hspace{-5cm}\leq C 2^{kp}2^{-lp}2^{-k\gamma}\int\limits_{cB_j}\fint\limits_{B(s,c2^{-k})}|f(s)-f(t)|^pd\mu(t)d\mu(s),\end{align*}

where we have used $cB_j\subset 19B_i\subset CB_j$. Now adding in $j:r_j\sim 2^{-k}$ and $k\leq l$, we get
	\begin{align*}II &\leq \sum_{k\leq l}\sum_{r_j\sim 2^{-k}}\int_{3B_j}\fint_{B(x,2^{-l})}|\mathcal{E}f(x)-\mathcal{E}f(y)|^p dm(y)dm(x)\\
	&\leq C2^{-lp}\sum_{k\leq l}2^{kp}2^{-k\gamma} \int_F\left(\sum_{r_j\sim 2^{-k}}\chi_{cB_j}(s)\right)\fint\limits_{B(s,c2^{-k})}|f(s)-f(t)|^pd\mu(t)d\mu(s)\\
	&\leq C2^{-lp}\sum_{k\leq l}2^{kp}2^{-k\gamma}E_pf(c2^{-k})^p.\end{align*}

Now for $I$, as $d(x,F)\lesssim 2^{-l}$ and $d(x,y)<2^{-l}$, we have $d(y,F)\lesssim 2^{-l}$ and there exist $k,m\geq l$ such that $d(x,F)\sim 2^{-k}, d(y,F)\sim 2^{-m}$ and, as $\sum_i\varphi_i(x)=\sum_j\varphi_j(y)=1$, we can write
\begin{align*}\mathcal{E}f(x)-\mathcal{E}f(y)=\sum_i\sum_j \varphi_i(x)\varphi_j(y)\fint_{19B_i}\fint_{19B_j}f(s)-f(t)d\mu(t)d\mu(s),\end{align*}
and for $d(x,F)\sim 2^{-k}, d(y,F)\sim 2^{-m}$, using that for $x\in 6B_i$ we get $B(x,c'2^{-k})\subset B_i\subset B(x,c2^{-k})$ and in a similar way for $y$,
\begin{align*}|\mathcal{E}f(x)-\mathcal{E}f(y)|^p &\leq C\sum_{r_i\sim 2^{-k}}\sum_{r_j\sim 2^{-m}} \chi_{6B_i}(x)\chi_{6B_j}(y)\fint_{B(x,c2^{-k})}\fint_{B(y,c2^{-m})}|f(s)-f(t)|^pd\mu(t)d\mu(s)\\
&\leq C\fint_{B(x,c2^{-k})}\fint_{B(y,c2^{-m})}|f(s)-f(t)|^pd\mu(t)d\mu(s).\end{align*}
Integrating first with respect to $y$,
	\begin{align*}\int_{y\in B(x,2^{-l}),d(y,F)\sim 2^{-m}}|\mathcal{E}f(x)-\mathcal{E}f(y)|^pdm(y) &\leq\\
	&\hspace{-3cm}\leq C\fint\limits_{B(x,c2^{-k})}\int\limits_{	y\in B(x,2^{-l}), d(y,F)\sim 2^{-m}	}\fint\limits_{B(y,c2^{-m})}|f(s)-f(t)|^pd\mu(t)dm(y)d\mu(s),\end{align*}
but
	\begin{align*}\int_{y\in B(x,2^{-l}),d(y,F)\sim 2^{-m}}\fint_{B(y,c2^{-m})}|f(s)-f(t)|^pd\mu(t)dm(y) &\leq\\
	&\hspace{-5cm}\leq \sum_{r_h\sim 2^{-m}}\int_{B(x,2^{-l})\cap cB_h}\fint_{B(y,c2^{-m})}|f(s)-f(t)|^pd\mu(t)dm(y)\\
	&\hspace{-5cm}\leq C\int_{B(x,2^{-l}+c2^{-m})}|f(s)-f(t)|^p\int_{B(t,c2^{-m})}\frac{\left(\sum_{r_j\sim 2^{-m}}\chi_{aB_j}(y)\right)}{\mu(B(y,c2^{-m}))}dm(y)d\mu(t)\\
	&\hspace{-5cm}\leq C2^{-m\gamma}\int_{B(x,c2^{-l})}|f(s)-f(t)|^pd\mu(t),\end{align*}
so we get
	\begin{align*}\int_{y\in B(x,2^{-l}),d(y,F)\sim 2^{-m}}|\mathcal{E}f(x)-\mathcal{E}f(y)|^pdm(y) &\leq C2^{-m\gamma}\fint\limits_{B(x,c2^{-k})}\int\limits_{B(x,c2^{-l})}|f(s)-f(t)|^pd\mu(t)d\mu(s);\end{align*}
and now integrating in $x$ over the strip $d(x,F)\sim 2^{-k}$,
	\begin{align*}\sum_{r_i\sim 2^{-k}}\int_{3B_i}\frac{1}{m(B(x,2^{-l}))}\int\limits_{y\in B(x,2^{-l}),d(y,F)\sim 2^{-m}}|\mathcal{E}f(x)-\mathcal{E}f(y)|^pdm(y)dm(x) &\leq\\
	&\hspace{-9cm}\leq C 2^{-m\gamma}\sum_{r_i\sim 2^{-k}}\int\limits_{3B_i}\frac{2^{l\gamma}}{\mu(B(x,c2^{-l}))}\fint\limits_{B(x,c2^{-k})}\int\limits_{B(x,c2^{-l})}|f(s)-f(t)|^pd\mu(t)d\mu(s)dm(x)\\
	&\hspace{-9cm}\leq C2^{-m\gamma}2^{l\gamma}\int\limits_F\fint\limits_{B(s,c2^{-l})}\frac{|f(s)-f(t)|^p}{\mu(B(s,c2^{-k}))}\int\limits_{B(s,c2^{-k})}\left(\sum_{r_i\sim 2^{-k}}\chi_{3B_i}(x)\right)dm(x)d\mu(t)\mu(s)\\
	&\hspace{-9cm}\leq C2^{-m\gamma}2^{l\gamma}2^{-k\gamma}\int_F\fint_{B(s,c2^{-l})}|f(s)-f(t)|^pd\mu(t)\mu(s)\\
	&\hspace{-9cm}=C2^{-m\gamma}2^{l\gamma}2^{-k\gamma}E_pf(c2^{-l})^p.\end{align*}
Finally, the bound for $I$ can be obtained by adding what we have just obtained in $k$ and $m$,
	\begin{align*}I &\leq \sum_{k\geq l}\int_{d(x,F)\sim 2^{-k}}\fint_{B(x,2^{-l})}|\mathcal{E}f(x)-\mathcal{E}f(y)|^p dm(y)dm(x)\\
	&\leq C2^{l\gamma}E_pf(2^{-l})^p\sum_{k\geq l}\sum_{m\geq l}2^{-m\gamma}2^{-k\gamma}\leq C2^{-l\gamma}E_pf(c2^{-l})^p.\end{align*}

This way, we have proved
\begin{align*}E_p\mathcal{E}f(2^{-l})&\leq C2^{-l}\left(\sum_{k\leq l}2^{k(p-\gamma)}E_pf(c2^{-k})^p\right)^{1/p}+C2^{-l}\|f\|_{p,\mu} +C2^{-l\gamma/p}E_pf(c2^{-l}).\end{align*}

Now, we have to consider each case separately. For $q<\infty$, by Hardy's inequality ($\nu=q/p$, applied to the first term of the sum) we have that
\begin{align*}\left(\sum_l 2^{l\alpha q} E_p\mathcal{E}f(2^{-l})^q\right)^{1/q} &\leq C\left(\sum_l 2^{l\beta q}2^{l\gamma q/p}2^{-lq}\left(\sum_{k\leq l}2^{k(p-\gamma)}E_pf(c2^{-k})^p\right)^{q/p}\right)^{1/q}\\
&\hspace{1cm}+C\left(\sum_l 2^{-l(1-\alpha)q}\|f\|_{p,\mu}^q\right)^{1/q}\\
&\hspace{1cm}+C\left(\sum_l 2^{l\beta q}2^{l\gamma q/p}2^{-l\gamma q/p}E_pf(c2^{-l})^q\right)^{1/q}\\
&\leq C\left(\sum_l 2^{l\beta q}2^{l\gamma q/p}2^{-lq}2^{l(q-\gamma q/p)}E_pf(c2^{-l})^q\right)^{1/q}\\
&\hspace{1cm}+C\|f\|_{p,\mu}\\
&\hspace{1cm}+C\left(\sum_l 2^{l\beta q}2^{l\gamma q/p}2^{-l\gamma q/p}E_pf(c2^{-l})^q\right)^{1/q}\\
&\leq C\|f\|_{p,\mu}+C\left(\sum_l 2^{l\beta q}E_pf(c2^{-l})^q\right)^{1/q}.\end{align*}

For $q=\infty$,
\begin{align*}2^{-l}\left(\sum_{k\leq l}2^{kp}2^{-k\gamma}E_pf(c2^{-k})^p\right)^{1/p} &\leq 2^{-l}\left(\sup_k 2^{k\beta p}E_pf(c2^{-k})^p\right)^{1/p}\left(\sum_{k\leq l}2^{kp}2^{-k\gamma}2^{-k\beta p}\right)^{1/p}\\
&\leq C2^{-l}2^{l(1-\alpha)}\left(\sup_k 2^{k\beta}E_pf(c2^{-k})\right)\\
&=C2^{-l\alpha }\left(\sup_{k\leq l}2^{k\beta}E_pf(c2^{-k})\right),\end{align*}
so for every $l\geq 1$,
\begin{align*}2^{l\alpha p}E_p\mathcal{E}f(2^{-l})^p\leq C\sup_{k\leq l} 2^{k\beta}E_pf(c2^{-k})+C\|f\|_{p,\mu}+C[f]_{B^\beta_{p,\infty}}.\end{align*}
\end{proof}

\section{Interpolation results}

In this section we prove, under certain conditions, that the interpolated space $\left(L^{\alpha,p},L^{\beta,p}\right)_{\theta,q}$, where $L^{\alpha,p}$ and $L^{\beta,p}$ are the potential spaces described in Section \ref{bess} and defined in \cite{M}, coincides with the Besov space $B^\gamma_{p,q}$, with $\gamma=\alpha+\theta(\beta-\alpha)$.

Let now $(X,d,m)$ be Ahlfors $N$-regular with $m(X)=\infty$. We aim to show $\left(L^{\alpha,p},L^{\beta,p}\right)_{\theta,q}=B^\gamma_{p,q}$ by proving two (continuous) inclusions: first we see that
\begin{align*}\left(L^{\alpha,p},L^{\beta,p}\right)_{\theta,q}\embeds B^\gamma_{p,q},\end{align*}
for $0<\alpha,\beta<1$, $1\leq p\leq\infty$, $1\leq q<\infty$ satisfying the relation mentioned above, using the $K$ functional described in Section \ref{jotaka}.

To check the other inclusion
\begin{align*}B^\gamma_{p,q}\embeds \left(L^{\alpha,p},L^{\beta,p}\right)_{\theta,q},\end{align*}
we prove
\begin{align*}\|f\|_{\theta,q,J}\leq C\|f\|_{B^\gamma_{p,q}},\end{align*}
where $J$ is the $J$-functional (also mentioned in \ref{jotaka}).

The embedding $\left(L^{\alpha,p},L^{\beta,p}\right)_{\theta,q}\embeds B^\gamma_{p,q}$ is a consequence of the fact that $L^{\alpha,p}\embeds B^\alpha_{p,\infty}$ for $0<\alpha<1$ (and the same holds for $\beta$), as mentioned in Section \ref{bess}, so by interpolation
\begin{align*}\left(L^{\alpha,p},L^{\beta,p}\right)_{\theta,q}\embeds \left(B^\alpha_{p,\infty},B^\beta_{p,\infty}\right)_{\theta,q}.\end{align*}

Then the result follows from the next lemma.

\begin{lem}\label{besovinf}For $\alpha\neq\beta>0$, $1\leq p,q\leq\infty$, $0<\theta<1$ and $\gamma=(1-\theta)\alpha+\theta\beta$,
\begin{align*}(B^\alpha_{p,\infty},B^\beta_{p,\infty})_{\theta,q}\embeds B^\gamma_{p,q}.\end{align*}\end{lem}
\begin{proof} This is based on the proof of Theorem 5.6.1 in \cite{BL}. First we show that
\begin{align*}t^{-\gamma}E_pf(t)\leq Ct^{-\theta(\beta-\alpha)}Kf(t^{\beta-\alpha}).\end{align*}
For this, let $A=\{s: t\leq s^{\beta-\alpha}<2t\}$, then in $A$ we have
\begin{align*}t^{-\gamma}E_pf(t)\leq Ct^{-\theta(\beta-\alpha)}s^{-\alpha(\beta-\alpha)}E_pf(s^{\beta-\alpha}),\end{align*}
and if $f=g+h$, as in $A$ we have $s^{-\alpha(\beta-\alpha)}\sim t^{\beta-\alpha}s^{-\beta(\beta-\alpha)}$,
\begin{align*}t^{-\gamma}E_pf(t)\leq Ct^{-\theta(\beta-\alpha)}\left(s^{-\alpha(\beta-\alpha)}E_pg(s^{\beta-\alpha})+t^{\beta-\alpha}s^{-\beta(\beta-\alpha)}E_ph(s^{\beta-\alpha})\right),\end{align*}
and from this inequality,
\begin{align*}t^{-\gamma}E_pf(t) &\leq Ct^{-\theta(\beta-\alpha)}\left([g]_{B_{p,\infty}^\alpha}+t^{\beta-\alpha}[h]_{B_{p,\infty}^\beta}\right)\\
&\leq Ct^{-\theta(\beta-\alpha)}\left(\|g\|_{B_{p,\infty}^\alpha}+t^{\beta-\alpha}\|h\|_{B_{p,\infty}^\beta}\right)\end{align*}
taking the infimum over all decompositions of $f$ we obtain what we wanted. From this, and from the fact that $\min(1,t)\|f\|_p\leq Kf(t)$, we conclude
\begin{align*}\|f\|_{B^\gamma_{p,q}}\leq C\|f\|_{\theta,q,K}.\end{align*}
\end{proof}

To prove the other embedding, we follow the ideas in \cite{P} for $\RR^n$ but in the continuous case (instead of the discrete one), where J. Peetre gives an adequate decomposition of functions in $B^\gamma_{p,q}$ in terms of functions in $\mathcal{L}^{\alpha,p}\cap\mathcal{L}^{\beta,p}$, and each piece satisfies certain bounds in the norms of $\mathcal{L}^{\alpha,p}$ and $\mathcal{L}^{\beta,p}$. This is done in this classical case using the fact that the Bessel potential $\mathcal{J}_\alpha$ is inversible and $\mathcal{J}_\alpha^{-1} \approx I+\mathcal{D}_\alpha$ (and the same for $\beta$). As mentioned in Theorem \ref{caract}, this same result holds in Ahlfors spaces for the case $0<\alpha,\beta<\alpha_0$ and $1<p<\infty$.

We now prove a theorem that will give us the other inclusion, assuming an appropiate decomposition exists.

\begin{thm}\label{fuerte}
Let $0<\alpha,\beta,\gamma,\theta<1$ with $\gamma=\alpha+\theta(\beta-\alpha)$ and $1\leq p,q\leq\infty$. Assume there exists a decomposition of the identity in terms of operators $(A_t)_{t>0}$ for functions $f\in B^\gamma_{p,q}$ as $A_tf\in L^{\alpha,p}\cap L^{\beta,p}$ satisfying
\begin{align*}f=\int_0^\infty A_t(f)\frac{dt}{t}=\int_0^1 A_tf\frac{dt}{t}+Af\end{align*}
and
\begin{align*}\|A_t(f)\|_{\alpha,p}\leq Ct^{-\alpha}E_pf(t),\quad 0<t<1;\end{align*}
\begin{align*}\|Af\|_{\alpha,p}\leq C\|f\|_p\end{align*}
and similar bounds for $\beta$.

Then we have
\begin{align*}t^{-\theta(\beta-\alpha)}J(A_t(f))(t^{\beta-\alpha})\leq Ct^{-\gamma}E_pf(t), \quad 0<t<1;\end{align*}
\begin{align*}J(Af)(1)\leq C\|f\|_p\end{align*}
and we conclude
\begin{align*}B^\gamma_{p,q}\embeds\left(L^{\alpha,p},L^{\beta,p}\right)_{\theta,q}.\end{align*}
\end{thm}
\begin{proof}
The first conclusion is immediate by definition of the $J$-functional:
\begin{align*}J(A_t(f))(s)=\max(\|A_t(f)\|_{\alpha,p},s\|A_t(f)\|_{\beta,p}),\end{align*}
then by hypothesis
\begin{align*}J(A_t(f))(t^{\beta-\alpha})\leq C\max\left(t^{-\alpha}E_pf(t),t^{\beta-\alpha}t^{-\beta}E_pf(t)\right)=Ct^{-\alpha}E_pf(t).\end{align*}
The second one is also trivial,
\begin{align*}J(Af)(1)=\max(\|Af\|_{\alpha,p},\|Af\|_{\beta,p})\leq C\|f\|_p.\end{align*}

Finally, assume $\alpha<\beta$ and observe
\begin{align*}f=Af+\int_0^1 \frac{1}{\beta-\alpha}A_{t^{1/(\beta-\alpha)}}(f)\frac{dt}{t}=Af+\int_0^1 \tilde{A}_{t}(f)\frac{dt}{t}\end{align*}
and therefore
\begin{align*}\|f\|_{\theta,q,J} &\leq J(Af)(1)+ \left(\int_0^1 t^{-\theta(\beta-\alpha)q}J(\tilde{A}_{t^{\beta-\alpha}}(f))(t^{\beta-\alpha})^q\frac{dt}{t}\right)^{1/q}\\
&\leq J(Af)(1)+C\left(\int_0^\infty t^{-\theta(\beta-\alpha)q}J(A_t(f))(t^{\beta-\alpha})^q\frac{dt}{t}\right)^{1/q}\\
&\leq C\|f\|_p+C\left(\int_0^\infty t^{-\gamma q}E_pf(t)^q\frac{dt}{t}\right)^{1/q}\leq C\|f\|_{B^\gamma_{p,q}}.\end{align*}
\end{proof}

We will spend the rest of the section proving that, under the assumption $\alpha,\beta<\alpha_0$, Calderón's reproducing formula
\begin{align*}f=\int_0^\infty \tilde{Q}_tQ_tf\frac{dt}{t}=\int_0^1 \tilde{Q}_tQ_tf\frac{dt}{t}+\mathbf{Q}f\end{align*}
gives the decomposition we want.

\begin{thm}\label{cotas}For $1<p<\infty$, $0<\alpha<1$ and $0<t<1$ we have
	\begin{align*}\|(I+D_\alpha)\tilde{Q}_tQ_tf\|_p\leq Ct^{-\alpha}E_pf(ct);\quad \|(I+D_\alpha)\mathbf{Q}f\|_p\leq C\|f\|_p.\end{align*}
\end{thm}
\begin{proof}

We divide the proof in three lemmas.

\begin{lem}Let $1<p<\infty$, $0<\alpha<1$ and $s>0$,
\begin{align*}\|D_\alpha \tilde{Q}_sQ_sf\|_p\leq Cs^{-\alpha}\|Q_sf\|_p.\end{align*}
\end{lem}
\begin{proof} With the formula for $D_\alpha$ in terms of $Q_t$
\begin{align*}D_\alpha g(x)=\int_0^\infty t^{-\alpha}Q_tg(x)\frac{dt}{t},\end{align*}
we have
\begin{align*}D_\alpha \tilde{Q}_sQ_sf(x) &=\int_0^\infty t^{-\alpha}Q_t\tilde{Q}_sQ_sf(x)\frac{dt}{t}\\
&=\int_0^\infty\int_X\int_X t^{-\alpha}q(x,y,t)\tilde{q}(y,z,s)Q_sf(z)dm(z)dm(y)\frac{dt}{t}\\
&=\int_X r_\alpha(x,z,s)Q_sf(z)dm(z),\end{align*}
where
\begin{align*}r_\alpha(x,z,s) &=\int_0^\infty\int_X t^{-\alpha}q(x,y,t)\tilde{q}(y,z,s)dm(y)\frac{dt}{t}\\
&=\int_0^{cs}\int_X t^{-\alpha}q(x,y,t)(\tilde{q}(y,z,s)-\tilde{q}(x,z,s))dm(y)\frac{dt}{t}\\
&\hspace{1cm}+\int_{cs}^\infty\int_X t^{-\alpha}q(x,y,t)\tilde{q}(y,z,s)dm(y)\frac{dt}{t}\\
&=I+II.\end{align*}
For the first part, as the term being integrated is nonzero only if $d(x,y)<ct<cs<C(s+d(y,z))$,
\begin{align*}I\leq Cs^\xi\int_0^{cs}\int_X t^{-\alpha} |q(x,y,t)|\frac{d(x,y)^\xi}{(s+d(y,z))^{N+2\xi}}dm(y)\frac{dt}{t},\end{align*} 
and for the second part,
\begin{align*}II\leq Cs^\xi\int_{cs}^\infty\int_X t^{-\alpha} |q(x,y,t)|\frac{1}{(s+d(y,z))^{N+\xi}}dm(y)\frac{dt}{t}.\end{align*} 
Then we get the integrals
\begin{align*}\int_X |r_\alpha(x,z,s)|dm(z) &\leq Cs^\xi\left(\int_0^{cs}\int_X t^{-\alpha} |q(x,y,t)|d(x,y)^\xi\frac{1}{s^{2\xi}}dm(y)\frac{dt}{t}\right.\\
	&\hspace{1cm}+\left.\int_{cs}^\infty\int_X t^{-\alpha} |q(x,y,t)|\frac{1}{s^\xi}dm(y)\frac{dt}{t}\right)\\
	&\leq Cs^\xi\left(\int_0^{cs}t^{-\alpha}t^\xi\frac{1}{s^{2\xi}}\frac{dt}{t}+\int_{cs}^\infty t^{-\alpha}\frac{1}{s^\xi}\frac{dt}{t}\right)=Cs^{-\alpha},\end{align*}
and
\begin{align*}\int_X |r_\alpha(x,z,s)|dm(x) &\leq Cs^\xi\left(\int_0^{cs}\int_X t^{-\alpha}t^\xi\frac{1}{(s+d(y,z))^{N+2\xi}}dm(y)\frac{dt}{t}\right.\\
	&\hspace{1cm}\left.+\int_{cs}^\infty\int_X t^{-\alpha} \frac{1}{(s+d(y,z))^{N+\xi}}dm(y)\frac{dt}{t}\right)\\
	&\leq Cs^\xi\left(\int_0^{cs}t^{-\alpha}t^\xi\frac{1}{s^{2\xi}}\frac{dt}{t}+\int_{cs}^\infty t^{-\alpha} \frac{1}{s^\xi}\frac{dt}{t}\right)=Cs^{-\alpha}.\end{align*}
Finally,
\begin{align*} |D_\alpha\tilde{Q}_sQ_sf(x)| &\leq \\
&\hspace{-1cm}\leq\left(\int_X |r_\alpha(x,z,s)|dm(z)\right)^{1/p'}\left(\int_X|r_\alpha(x,z,s)||Q_sf(z)|^pdm(z)\right)^{1/p}\\
&\hspace{-1cm}\leq Cs^{-\alpha/p'}\left(\int_X|r_\alpha(x,z,s)||Q_sf(z)|^pdm(z)\right)^{1/p}\end{align*}
and therefore
\begin{align*}\int_X|D_\alpha\tilde{Q}_sQ_sf(x)|^pdm(x) &\leq\\
&\hspace{-1cm}\leq Cs^{-\alpha(p/p')}\int_X\int_X |r_\alpha(x,z,s)||Q_sf(z)|^pdm(z)dm(x)\\
&\hspace{-1cm}\leq Cs^{-\alpha(p/p')}s^{-\alpha}\int_X|Q_sf(z)|^pdm(z)\\
&\hspace{-1cm}=Cs^{-\alpha p}\|Q_sf\|_p^p.\end{align*}
\end{proof}

\begin{lem}If $1<p<\infty$ and $t>0$, 
\begin{align*}\|Q_tf\|_p\leq CE_pf(4t).\end{align*}\end{lem}
\begin{proof}As $\int_X q(x,y,t)f(x)dm(y)=0$,
\begin{align*}|Q_tf(x)|=\left|\int_X q(x,y,t)f(y)dm(y)\right|=\left|\int_X q(x,y,t)(f(y)-f(x))dm(y)\right|,\end{align*}
then using the properties of the kernel $q$ mentioned in Section \ref{aprox},
\begin{align*}|Q_tf(x)|^p &\leq \left(\int_X |q(x,y,t)||f(y)-f(x)|^pdm(y)\right)\left(\int_X |q(x,y,t)|dm(y)\right)^{p/p'}\\
&\leq \frac{C}{t^N}\int_{B(x,4t)}|f(y)-f(x)|^pdm(y).\end{align*}
We obtain
\begin{align*}\|Q_tf\|_p^p\leq C\int_X\fint_{B(x,4t)}|f(y)-f(x)|^pdm(y)dm(x)\leq CE_pf(4t)^p.\end{align*}
\end{proof}

The third and final lemma is the longest, for it requires the $T1$ theorem to prove boundedness in $L^p$.

\begin{lem}If $1<p<\infty$,
\begin{align*}\|\mathbf{Q}f\|_p\leq C\|f\|_p.\end{align*}\end{lem}
\begin{proof}We have 
\begin{align*}\int_1^\infty \tilde{Q}_tQ_tf(x)\frac{dt}{t}=\int_X \left(\int_1^\infty\int_X \tilde{q}(x,y,t)q(y,z,t)dm(y)\frac{dt}{t}\right)f(z)dm(z),\end{align*}
then if we prove that
\begin{align*}\mathbf{Q}(x,z)=\int_1^\infty\int_X \tilde{q}(x,y,t)q(y,z,t)dm(y)\frac{dt}{t}\end{align*}
is a standard kernel and the $T1$ theorem holds for the operator $\mathbf{Q}$, we will get our result.
\begin{itemize}
	\item Size of the kernel $\mathbf{Q}$: as $d(y,z)<ct<c(t+d(x,z))$ where $q(y,z,t)\neq0$,
	\begin{align*}|\mathbf{Q}(x,z)| &\leq \left|\int_1^\infty\int_X (\tilde{q}(x,y,t)-\tilde{q}(x,z,t))q(y,z,t)dm(y)\frac{dt}{t}\right|\\
	&\leq C\int_1^\infty\int_{d(y,z)<ct} \frac{t^\xi d(y,z)^\xi}{(t+d(x,z))^{N+2\xi}}\frac{1}{t^N}dm(y)\frac{dt}{t}\\
	&\leq C\int_1^\infty \frac{t^{2\xi}}{(t+d(x,z))^{N+2\xi}}\frac{dt}{t}\leq \int_1^{d(x,z)}+\int_{d(x,z)}^\infty\\
	&\leq C\frac{1}{d(x,z)^{N+2\xi}}\int_1^{d(x,z)}t^{2\xi}\frac{dt}{t}+C\int_{d(x,z)}^\infty \frac{1}{t^N}\frac{dt}{t}\\
	&\leq C\frac{1}{d(x,z)^N.}\end{align*}
	
	\item Regularity (1): for $d(x,x')< Cd(x,z)$, then $d(x,x')<Cd(x,z)\leq C(d(x,y)+d(y,z))\leq C(t+d(x,y))$,
	\begin{align*}|\mathbf{Q}(x,z)-\mathbf{Q}(x',z)| &\leq \int_1^\infty\int_{d(y,z)<4t}|\tilde{q}(x,y,t)-\tilde{q}(x',y,t)|\frac{1}{t^N}dm(y)\frac{dt}{t}\\
	&\hspace{-1.5cm}\leq Cd(x,x')^\xi\int_1^\infty \frac{1}{t^{N-\xi}} \int_{d(y,z)<4t}\frac{1}{(t+d(x,y))^{N+2\xi}}dm(y)\frac{dt}{t}\\
	&\hspace{-1.5cm}\leq \int_0^{d(x,z)/5}+\int_{d(x,z)/5}^\infty = I+II.\end{align*}
	For $I$, as $d(x,z)>5t>4t>d(y,z)$, we have $d(x,z)>\frac{5}{4}d(y,z)$ and therefore $d(x,y)>\frac{1}{5}d(x,z)$, so
	\begin{align*}I\leq C\frac{d(x,x')^\xi}{d(x,z)^{N+2\xi}}\int_0^{d(x,z)/5}\frac{1}{t^{N-\xi}}t^N\frac{dt}{t}\leq C\frac{d(x,x')^\xi}{d(x,z)^{N+\xi}};\end{align*}
	and for $II$,
	\begin{align*}II\leq Cd(x,x')^\xi\int_{d(x,z)/5}^\infty \frac{1}{t^{N-\xi}}\frac{t^N}{t^{N+2\xi}}\frac{dt}{t}\leq C\frac{d(x,x')^\xi}{d(x,z)^{N+\xi}}.\end{align*}
	
	\item Regularity (2): if $d(z,z')\leq Cd(x,z)$, then
	\begin{align*}|\mathbf{Q}(x,z)-\mathbf{Q}(x,z')| &= \left|\int_1^\infty\int_X\tilde{q}(x,y,t)(q(y,z,t)-q(y,z',t))dm(y)\frac{dt}{t}\right|\\
	&\hspace{-3cm}\leq \int_{d(z,z')/C}^\infty\left|\int_X (\tilde{q}(x,y,t)-\tilde{q}(x,z,t))(q(y,z,t)-q(y,z',t))dm(y)\right|\frac{dt}{t}\\
	&\hspace{-2cm}+ \int_0^{d(z,z')/C}\left|\int_X (\tilde{q}(x,y,t)-\tilde{q}(x,z,t))q(y,z,t)dm(y)\right|\frac{dt}{t}\\
	&\hspace{-2cm}+ \int_0^{d(z,z')/C}\left|\int_X (\tilde{q}(x,y,t)-\tilde{q}(x,z',t))q(y,z',t)dm(y)\right|\frac{dt}{t}\\
	&\hspace{-3cm}= I + II + III.\end{align*}
	
	For $I$, we have $d(y,z)\leq C(t+d(x,z))$ so
	\begin{align*}I &\leq C\int_{d(z,z')/C}^\infty\int_{B(z,Ct)}\frac{t^\xi d(y,z)^\xi}{(t+d(x,z))^{N+2\xi}}\frac{d(z,z')}{t^{N+1}}dm(y)\frac{dt}{t}\\
	&\leq Cd(z,z')\int_{d(z,z')/C}^\infty\frac{t^\xi}{(t+d(x,z))^{N+2\xi}}\frac{t^{N+\xi}}{t^{N+1}}\frac{dt}{t}\\
	&= Cd(z,z')\int_{d(z,z')/C}^\infty\frac{t^{2\xi-1}}{(t+d(x,z))^{N+2\xi}}\frac{dt}{t}\end{align*}
	and as $d(z,z')/C\leq d(x,z)$ and we can take $\xi>1/2$,
	\begin{align*}I &\leq Cd(z,z')\left(\frac{1}{d(x,z)^{N+2\xi}}\int_0^{d(x,z)}t^{2\xi-1}\frac{dt}{t}+\int_{d(x,z)}^\infty \frac{1}{t^{N+1}}\frac{dt}{t}\right)\\
	&\leq C\frac{d(z,z')}{d(x,z)^{N+1}}.\end{align*}
	
	We now bound $II$, $III$ is similar. As $d(y,z)<4t$ we get $d(y,z)\leq C(t+d(x,z))$, then
	\begin{align*}II &\leq C \int_0^{d(z,z')/C}\int_{B(z,4t)}\frac{t^\xi d(y,z)^\xi}{(t+d(x,z))^{N+2\xi}}\frac{1}{t^N}dm(y)\frac{dt}{t}\\
	&\leq C\frac{1}{d(x,z)^{N+2\xi}}\int_0^{d(z,z')/C}t^{2\xi}\frac{dt}{t}\\
	&\leq C\frac{d(z,z')^{2\xi}}{d(x,z)^{N+2\xi}}\leq C\frac{d(z,z')}{d(x,z)^{N+1}},\end{align*}
	where the last inequality also requires $\xi>1/2$.
		
	\item $\mathbf{Q}1=\mathbf{Q}^*1=0$. One is immediate,
	\begin{align*}\mathbf{Q}1=\int_X\int_1^\infty \int_X \tilde{q}(x,y,t)q(y,z,t)dm(y)\frac{dt}{t}dm(z)=0.\end{align*}
	
	For the other one, define $\mathbf{Q}^*(z,x)=\mathbf{Q}(x,z)$, then
	\begin{align*}<\mathbf{Q}f,g> &=\int_X\left(\int_X \mathbf{Q}(x,z)f(z)dm(z)\right)g(x)dm(x)\\
	&=\int_X\left(\int_X \mathbf{Q}^*(z,x)g(x)dm(x)\right)f(z)dm(z)=<f,\mathbf{Q}^*g>\end{align*}
	so clearly $\mathbf{Q}^*1=0$.
	
	\item Weak boundedness: let $f,g\in C_c^\gamma(B)$ for $\gamma>0$ and $B$ a ball. 
	\begin{align*}<\mathbf{Q}f,g> &=\int_X\int_X\int_1^\infty\int_X \tilde{q}(x,y,t)q(y,z,t)f(z)g(x)dm(y)\frac{dt}{t}dm(z)dm(x)\\
	&=\int\limits_1^{Cdiam(B)}\int_X\int_X\int_X \tilde{q}(x,y,t)q(y,z,t)f(z)g(x)dm(y)dm(z)dm(x)\frac{dt}{t}\\
	&\hspace{1cm}-\int\limits_1^{Cdiam(B)}\int_X\int_X\int_X \tilde{q}(x,y,t)q(y,z,t)f(z)g(y)dm(y)dm(z)dm(x)\frac{dt}{t}\\
	&\hspace{1cm}+\int\limits_{Cdiam(B)}^\infty\int_X\int_X\int_X \tilde{q}(x,y,t)q(y,z,t)f(z)g(x)dm(y)dm(z)dm(x)\frac{dt}{t}.\end{align*}
	
	For $t<Cdiam(B)$, $d(x,y)<ct$,
	\begin{align*}\left|\int_X\int_X\int_X \tilde{q}(x,y,t)q(y,z,t)f(z)(g(x)-g(y))dm(y)dm(z)dm(x)\right| &\leq\\
	&\hspace{-9cm}\leq C\|f\|_\infty [g]_\gamma \int_X\int_B\int_{B(z,4t)}\frac{t^\xi}{(t+d(x,y))^{N+\xi}}\frac{1}{t^N}d(x,y)^\gamma dm(y)dm(z)dm(x)\\
	&\hspace{-9cm}\leq C\|f\|_\infty [g]_\gamma m(B)t^\gamma\\
	&\hspace{-9cm}\leq C[f]_\gamma [g]_\gamma diam(B)^{N+\gamma}t^\gamma.\end{align*}
	And for $t\geq Cdiam(B)$,
	\begin{align*}\left|\int_X\int_X\int_X \tilde{q}(x,y,t)q(y,z,t)f(z)g(x)dm(y)dm(z)dm(x)\right| &\leq\\
	&\hspace{-7cm}\leq C\|f\|_\infty\|g\|_\infty\int_B\int_B \int_{B(z,4t)}\frac{t^\xi}{(t+d(x,y))^{N+\xi}}\frac{1}{t^N}dm(y)dm(z)dm(x)\\
	&\hspace{-7cm}\leq C\|f\|_\infty\|g\|_\infty \frac{1}{t^N}m(B)^2\\
	&\hspace{-7cm}\leq C[f]_\gamma [g]_\gamma diam(B)^{2N+2\gamma}t^{-N}.\end{align*}
	We conclude
	\begin{align*}|<\mathbf{Q}f,g>| &\leq C[f]_\gamma [g]_\gamma diam(B)^{N+\gamma}\int_0^{Cdiam(B)}t^\gamma\frac{dt}{t}\\
	&\hspace{1cm}+C[f]_\gamma [g]_\gamma diam(B)^{2N+2\gamma}\int_{Cdiam(B)}^\infty t^{-N}\frac{dt}{t}\\
	&\leq C[f]_\gamma [g]_\gamma diam(B)^{N+2\gamma}.\end{align*}
\end{itemize}\end{proof}

From this three lemmas, the theorem follows, as

\begin{align*}\|(I+D_\alpha)\tilde{Q}_tQ_tf\|_p &\leq \|\tilde{Q}_tQ_tf\|_p+Ct^{-\alpha}E_pf(ct)\\
&\leq CE_pf(ct)+Ct^{-\alpha}E_pf(ct)\\
&\leq Ct^{-\alpha}E_pf(ct),\end{align*}
where we have used the first and second lemmas, and the fact that $\tilde{Q}_t$ is bounded in $L^p$. On the other hand,

\begin{align*}\|(I+D_\alpha)\mathbf{Q}f\|_p &\leq C\|f\|_p+\int_1^\infty \|D_\alpha \tilde{Q}_tQ_tf\|_p\frac{dt}{t}\\
&\leq C\|f\|_p+C\int_1^\infty t^{-\alpha}\|Q_tf\|_p\frac{dt}{t}\\
&\leq C\|f\|_p+C\|f\|_p\int_1^\infty t^{-\alpha}\frac{dt}{t}\\
&\leq C\|f\|_p,\end{align*}
where we have used the third lemma and Minkowski's integral inequality, plus the first lemma and the fact that $Q_t$ is bounded in $L^p$.
\end{proof}

By the previous theorem and the fact that $\|g\|_{\alpha,p}\sim \|(I+D_\alpha)g\|_p$ for $\alpha<\alpha_0$ (see Theorem \ref{caract}), we can take $A_t=\tilde{Q}_tQ_t$ and $A=\mathbf{Q}$ in Theorem \ref{fuerte} and conclude the following. 

\begin{thm}[\textbf{Interpolation theorem for potential spaces}] For $1<p,q<\infty$, $0<\alpha,\beta<\alpha_0$, $0<\theta<1$ and $\gamma=\alpha+\theta(\beta-\alpha)$, we get
\begin{align*}B^\gamma_{p,q}=(L^{\alpha,p},L^{\beta,p})_{\theta,q}.\end{align*}\end{thm}

\section{Restriction theorems}

For the restriction theorem, we will use the result from the previous section. For that let $(X,d,m)$ be Ahlfors $N$-regular with $m(X)=\infty$ and let $F\subset X$ be a closed subset with $\mu$ a Borel measure supported in $F$ such that $(F,d,\mu)$ is Ahlfors $d$-regular, $0<d<N$.

In a way analogous to \cite{JW}, to prove the restriction theorem we first show that potential spaces $L^{\alpha,p}$ leave a trace in $F$, belonging to a certain Besov space. The theorem then will be a consequence of this and of the interpolation theorem from the last section.

Let $k_\alpha$ be the kernel for $J_\alpha$. To prove functions in $L^{\alpha,p}(X)=J_\alpha(L^p(X))$ belong to a Besov space when restricted to $F$, we need an estimate of the size and smothness of integrals of $k_\alpha$ in $F$. We do this in the following lemma, which has a proof that results from size and smoothness estimates for $k_\alpha$.

\begin{lem}\label{besovk2}If $0<\alpha<1$ and $q>0$ then:
\begin{enumerate}
	\item If $q(N-\alpha)<d<q(N+\alpha)$, there exists $C>0$ such that for $z\in X$ we have
	\begin{align*}\int_F k_\alpha(s,z)^qd\mu(s)\leq C<\infty.\end{align*}
	\item If $q(N-\alpha)<d<q(N-\alpha+1)$ and $0<r<Cdiam(F)$, there exists $C>0$ such that for $z\in X$ we have
	\begin{align*}\int_F\fint_{B(s,r)}|k_\alpha(s,z)-k_\alpha(t,z)|^qd\mu(t)d\mu(s)\leq C r^{d-q(N-\alpha)}.\end{align*}
\end{enumerate}
\end{lem}

We show now a restriction theorem for Sobolev functions, that we will use to prove the theorem for Besov functions, but that has importance on its own.

\begin{thm}[\textbf{Restriction theorem for $L^{\alpha,p}$}] Let $0<\alpha<1$ and $1<p<\infty$ satisfying $0<\beta=\alpha-\frac{N-d}{p}<1$. Then there exists a continuous linear operator
\begin{align*}\mathcal{R}: L^{\alpha,p}(X,m)\imply B^\beta_{p,\infty}(F,\mu)\end{align*}
satisfying $\mathcal{R}f=f|_F$ for continuous functions.\end{thm}
\begin{proof}
Define for continuous functions the operator $\mathcal{R}f=f|_F$. We will show that for $f$ continuous with $f=J_\alpha g, g\in L^p(X)$,
\begin{enumerate}
	\item \begin{align*}\int_F |f(s)|^pd\mu(s)\leq C\|g\|^p_p.\end{align*}
	\item if $0<r<Cdiam(F)$,
	\begin{align*}\int_F\fint_{B(s,r)}|f(s)-f(t)|^pd\mu(t)d\mu(s)\leq C r^{\beta p}\|g\|^p_p.\end{align*}
\end{enumerate}
Then
\begin{align*}\|\mathcal{R}f\|_{B^\beta_{p,q}(F,\mu)}\leq C\|f\|_{L^{\alpha,p}(X,m)},\end{align*}
and as continuous functions are dense in $L^{\alpha,p}$ (see Section \ref{bess}), we can conclude the theorem.

We proceed now to prove both inequalities. 
\begin{enumerate}
	\item If we take $0<a<1$,
	\begin{align*}\int_F |f(s)|^pd\mu(s) &\leq \int_F \left(\int_X k_\alpha(s,y)^{ap}|g(y)|^pdm(y)\right)\left(\int_X k_\alpha(s,y)^{(1-a)p'}dm(y)\right)^{p/p'}d\mu(s)\\
&\leq C\int_X \left(\int_F k_\alpha(s,y)^{ap}d\mu(s)\right)|g(y)|^pdm(y)\\
&\leq C\|g\|_p^p.\end{align*}
This follows as long as $0<a<1$ satisfies $(1-a)p'(N-\alpha)<N<(1-a)p'(N+\alpha)$ and $ap(N-\alpha)<d<ap(N+\alpha)$, or equivalently
\begin{align*}\frac{d}{p(N+\alpha)}<a<\frac{d}{p(N-\alpha)}\end{align*}
and
\begin{align*}\frac{N-\alpha p}{p(N-\alpha)}<a<\frac{N+\alpha p}{p(N+\alpha)},\end{align*}
and this always can be done, as $d<N+\alpha p$ and $N-\alpha p<d$.

	\item Once more take $0<a<1$. Given $s,t\in F$,
	\begin{align*}|f(s)-f(t)|^p &\leq\left(\int_X|k_\alpha(x,s)-k_\alpha(x,t)||g(x)|dm(x)\right)^p\\
&\leq\left(\int_X|k_\alpha(x,s)-k_\alpha(x,t)|^{ap}|g(x)|^pdm(x)\right)\left(\int_X|k_\alpha(x,s)-k_\alpha(x,t)|^{(1-a)p'}dm(x)\right)^{p/p'}.\end{align*}
Then if $d(s,t)<r$,
\begin{align*}\int_X|k_\alpha(x,s)-k_\alpha(x,t)|^{(1-a)p'}dm(x)\leq C r^{N-(1-a)p'(N-\alpha)}\end{align*}
if $0<a<1$ satisfies $(1-a)p'(N-\alpha)<N<(1-a)p'(N-\alpha+1)$. Now,
\begin{align*}\int_F\fint_{B(s,r)}|f(s)-f(t)|^pd\mu(t)d\mu(s) &\leq\\
&\hspace{-4cm}\leq C r^{N\frac{p}{p'}-(1-a)p(N-\alpha)}\int_X\left(\int_F\fint_{B(s,r)}|k_\alpha(x,s)-k_\alpha(x,t)|^{ap}d\mu(t)d\mu(s)\right)|g(x)|^pdm(x)\end{align*}
and therefore
\begin{align*}\int_F\fint_{B(s,r)}|k_\alpha(x,s)-k_\alpha(x,t)|^{ap}d\mu(t)d\mu(s)\leq C r^{d-ap(N-\alpha)}\end{align*}
whenever $a$ satisfies $ap(N-\alpha)<d<ap(N-\alpha+1)$ and we conclude
\begin{align*}\int_F\fint_{B(s,r)}|f(s)-f(t)|^pd\mu(t)d\mu(s) &\leq C r^{N\frac{p}{p'}-(1-a)p(N-\alpha)}r^{d-ap(N-\alpha)}\|g\|_p^p\\
&=C r^{p\beta}\|g\|_p^p.\end{align*}
The existence of this $a$ is equivalent to
\begin{align*}\frac{d}{p(N+1-\alpha)}<a<\frac{d}{p(N-\alpha)}\end{align*}
and
\begin{align*}\frac{N-\alpha p}{p(N-\alpha)}<a<\frac{N+p-\alpha p}{p(N+1-\alpha)}.\end{align*}
These intervals are well defined, as $p(N+1-\alpha)>p(N-\alpha)$ and
\begin{align*}(N-\alpha p)(N+1-\alpha) &=N^2-N\alpha p+N-\alpha p-N\alpha+\alpha^2p\\
&<N^2-N\alpha +Np-\alpha p-N\alpha p+\alpha^2p\\
&=(N-\alpha)(N+p-\alpha p),\end{align*}
and they overlap, as $N-\alpha p<d$ and $d<N+p-\alpha p$.
\end{enumerate}
\end{proof}

Using this theorem for $\alpha_1,\alpha_2$ and their corresponding $\beta_1,\beta_2$, and the fact that $\left(B_{p,\infty}^{\beta_1},B_{p,\infty}^{\beta_2}\right)_{\theta,q}\embeds B^\beta_{p,q}$ (see Lemma \ref{besovinf}), we obtain

\begin{coro}If $0<\alpha_1<\alpha_2<1$, then the restriction operator $\mathcal{R}$ is well defined and continuous
\begin{align*}\mathcal{R}:\left(L^{\alpha_1,p},L^{\alpha_2,p}\right)_{\theta,q}\imply B^\beta_{p,q}(F,\mu),\end{align*}
where $\beta=(1-\theta)\alpha_1+\theta\alpha_2-\frac{N-d}{p}$.
\end{coro}

Finally, as for the case $\alpha_1,\alpha_2<\alpha_0$ we have a characterization for $\left(L^{\alpha_1,p},L^{\alpha_2,p}\right)_{\theta,q}$, using the interpolation result from the previous section we can conclude the restriction theorem. 

\begin{thm}[\textbf{Restriction theorem for $B^\alpha_{p,q}(X,m)$}] If $\alpha<\alpha_0$, $1<p<\infty$ and $\beta=\alpha-\frac{N-d}{p}>0$, then for $1\leq q<\infty$ there exists a continuous linear operator
\begin{align*}\mathcal{R}: B^\alpha_{p,q}(X,m)\imply B^\beta_{p,q}(F,\mu)\end{align*}
such that $\mathcal{R} f=f|_F$ for $f$ continuous.\end{thm}

\section*{Acknowledgements}

The author is infinitely indebted to his advisors Eleonor `Pola' Harboure and Hugo Aimar for their guidance and support throughout the development of his doctoral thesis and its resulting papers.

\textit{E-mail address:} \texttt{mmarcos@santafe-conicet.gov.ar}

\end{document}